\documentclass[reqno]{amsart}
\usepackage{amsfonts, amsmath, amssymb, amsthm, color}
\usepackage[hmargin=3cm, vmargin=4cm]{geometry}

\newtheorem{thm}{Theorem}[section]
\newtheorem{lemma}[thm]{Lemma}
\newtheorem{prop}[thm]{Proposition}

\theoremstyle{definition}
\newtheorem{rmk}[thm]{Remark}

\newcommand{\ep}{\epsilon}
\newcommand{\ome}{\Omega_{\ep}}
\newcommand{\pome}{\partial\Omega_{\ep}}
\newcommand{\rn}{\mathbb{R}^n}

\newcommand{\inte}{\int_{\ome}}
\newcommand{\intwe}{\int_{\widetilde{\Omega}_\ep}}
\newcommand{\intal}{\int_{A_l}}

\newcommand{\home}{H^1_0(\ome)}

\newcommand{\udx}{U_{\delta,\xi}}
\newcommand{\pdxn}{\psi_{\delta, \xi}^0}
\newcommand{\pdxj}{\psi_{\delta, \xi}^j}

\newcommand{\md}{\mathbf{d}}
\newcommand{\ms}{\boldsymbol\sigma}
\newcommand{\mt}{\mathbf{t}}

\newcommand{\mca}{\mathcal{A}}
\newcommand{\mcb}{\mathcal{B}}
\newcommand{\mcd}{\mathcal{D}}

\newcommand{\sik}{\sum_{i=1}^k}
\newcommand{\slk}{\sum_{l=1}^k}

\renewcommand{\(}{\left(}
\renewcommand{\)}{\right)}

\newcommand{\pe}{^{\perp}}

\newcommand{\la}{\left\langle}
\newcommand{\ra}{\right\rangle}

\newcommand{\wb}{\widetilde{b}}
\newcommand{\wpe}{\widetilde{\Psi}}

\newcommand{\ps}{\partial_s}

\newcommand{\rr}{\mathbb{R}}

\begin{document}
\title{Supercritical problems in domains with thin toroidal holes}
\author{Seunghyeok Kim}
\address[Seunghyeok Kim] {Department of Mathematical Sciences, Korea Advanced Institute of Science and Technology, Daejeon 305-701, Republic of Korea}
\email{shkim0401@gmail.com}

\author{Angela Pistoia}
\address[Angela Pistoia] {Dipartimento SBAI, Universit\`{a} di Roma ``La Sapienza", via Antonio Scarpa 16, 00161 Roma, Italy}
\email{pistoia@dmmm.uniroma1.it}

\begin{abstract}
In this paper we study the  Lane-Emden-Fowler equation
$$ (P)_ \ep\quad \left\{
\begin{aligned}
&\Delta u+|u|^{q-2}u=0\ &\hbox{in}\ \mathcal D_ \ep,\\
& u=0\ &\hbox{on}\ \partial\mathcal D_ \ep.\\
\end{aligned}\right.
$$
Here $\mathcal D_ \ep=\mathcal D\setminus \left\{x\in \mathcal D\ :\ \mathrm{dist}(x,\Gamma_\ell)\le  \ep
\right\}$, $\mathcal D$ is a smooth bounded domain in $\rr^N$, $\Gamma_\ell$ is an $\ell-$dimensional closed manifold such that
$\Gamma_\ell\subset\mathcal D$ with $1\le  \ell\le N-3$
and $q={2(N-\ell)\over N-\ell-2} .$
We prove that, under some symmetry assumptions,  the number of sign changing solutions
to $ (P)_ \ep$  increases as $ \ep$ goes to zero.
\end{abstract}

\date{May 17, 2013.}
\subjclass[2010]{}
\thanks{}
\maketitle
\numberwithin{equation}{section}

\section{Introduction}
The paper deals with the Lane-Emden-Fowler equation
\begin{equation}\label{p0}
\left\{
\begin{aligned}
&\Delta u+|u|^{q-2}u=0\ &\hbox{in}\ \mathcal D,\\
& u=0\ &\hbox{on}\ \partial\mathcal D,\\
\end{aligned}\right.
\end{equation}
where $\mathcal D$ is a smooth bounded domain in $\rr^N$ and $q>2.$

A main characteristic of   problem \eqref{p0} is the role played by the $(\ell+1)-$critical exponent $2^*_{N,\ell} $  in the solvability question.
Here $
2^*_{N,\ell}:={2(N-\ell)\over N-\ell-2}$ where  $\ell$ is an integer between $0$ and $N-3.$
  $2^*_{N,\ell}$ is nothing but the  critical Sobolev exponent in dimension $N-\ell,$ i.e. $2^*_{N,\ell}= 2^*_{N-\ell,0}.$

If $q<2^*_{N,0}$ problem \eqref{p0} has one positive solution and infinitely many sign changing solutions. The proof relies on   standard variational arguments and
uses
the compactness of the embedding $H^1_0(\mathcal D)\hookrightarrow L^{q}(\mathcal D).$ When
$q\ge 2^*_{N,0}$ the compactness of the embedding is not true anymore and so existence of solutions becomes a delicate issue.
Pohozaev \cite{Po} discovered that no solution exists when the domain is star-shaped. On the other hand Kazdan-Warner \cite{KW} proved that if $\mathcal D$ is an annulus the compactness is restored in the class of radial functions and so problem \eqref{p0} has  one positive radial solution and infinitely many sign changing radial solutions for any $q$.
If
$q= 2^*_{N,0}$ Bahri and Coron \cite{BC} established the existence of at least
one positive solution to problem \eqref{p0} in every domain $\mathcal{D}$
having nontrivial reduced homology with $\mathbb{Z}/2$-coefficients. However, the topology
in the supercritical case   is not enough to guarantee existence. In fact,
for each $1\leq\ell\leq N-3,$ Passaseo \cite{Pa1,Pa2} exhibited domains having
the homotopy type of an $\ell$-dimensional sphere in which problem \eqref{p0}
does not have a nontrivial solution for $q\ge 2^*_{N,\ell}.$ Existence may fail even in domains with richer topology, as
shown by Clapp-Faya-Pistoia \cite{CFP}.

Many results about solvability of \eqref{p0} have been obtained when the exponent $q$ is close to    $2^*_{N,\ell}$ for some integer $\ell$. In particular, in this case
  it is possible to find  positive and sign changing  solutions which blows-up at $\ell-$dimensional manifolds as $q$ approaches $2^*_{N,\ell}$.
In the easiest case $\ell=0$ many authors have constructed positive and sign changing solutions which blows-up at one or more points in $\mathcal D$ as $q$ approaches the usual critical Sobolev exponent $2^*_{N,0}$ (see for example Bahri-Li-Rey \cite{BLR}, Bartsch-Micheletti-Pistoia \cite{BMP}, Musso-Pistoia \cite{MP}, Pistoia-Weth \cite{PW}, Del Pino-Felmer-Musso \cite{DFM}). We note that we could think at a point like a $0-$dimensional manifold!
When $\ell=1$  Del Pino-Musso-Pacard \cite{DMP}   built a positive solution to \eqref{p0} which blows-up at a suitable geodesic (i.e. $1-$dimensional manifold) of the boundary of $\mathcal D$ as $q$ approaches $2^*_{N,1}$.
Recently, positive and sign changing solutions blowing-up at  $\ell-$dimensional manifolds as $q$ approaches $2^*_{N,\ell}$
have been built  in domains involving symmetries by Ackermann-Clapp-Pistoia \cite{ACP} and Kim-Pistoia \cite{KP,KP2}.

  There are a few results about existence and multiplicity of solutions to problem \eqref{p0}  in the pure   supercritical case, i.e. $q=2^*_{N,\ell}.$   In particular,  Yan-Wei \cite{YW}   exhibited a torus-like domain in which problem \eqref{p0} has infinitely many positive solutions. Recently, Clapp-Faya-Pistoia \cite{CFP2} used  Hopf fibration to build a positive solution to problem \eqref{p0} when the domain $\mathcal D$ is an annulus in $\mathbb R^{2N}$ with a think spherical hole.

\medskip
In this paper we consider the supercritical problem \eqref{p0}   in a domain with an $\ell-$dimensional hole, namely
\begin{equation}\label{p1}
\left\{
\begin{aligned}
&\Delta u+|u|^{q-2}u=0\ &\hbox{in}\ \mathcal D_\ep,\\
 &u=0\ &\hbox{on}\ \partial\mathcal D_\ep,\\
\end{aligned}\right.
\end{equation}
where  $\mathcal D_ \ep:=\mathcal   D\setminus \left\{x\in \mathcal D\ :\ \mathrm {dist}(x,\Gamma_\ell)\le \ep\right\}$,   $\mathcal D$ is  a smooth bounded domain in $\rr^N,$
  $\Gamma_\ell$  is a closed $\ell-$dimensional manifold such that $\Gamma\subset \mathcal D,$    $\ep$ is small enough and $q\ge2^*_{N,\ell}$ .

If $\ell=0,$  the set $\Gamma_\ell$ reduces to a point $\xi_0\in\mathcal D$ and $\mathcal D_ \ep:=\mathcal   D\setminus B(\xi_0,\ep) $ is the Coron's type domain.
In this case, it is known that    problem   \eqref{p1}   has one positive solution  and an arbitrary large number of sign changing solutions whose number increases as $\ep$ goes to zero, for almost all the exponents $q$'s.
The critical case
$q=2^*_{N,0}$ has been studied by Coron \cite{Co}, Musso-Pistoia \cite{MP2} and Ge-Musso-Pistoia \cite{GMP}.
When
  $q> 2^*_{N,0}$  and $q$ is different from a resonant sequence $q_j\nearrow+\infty $,   the result has been obtained  by
 Del Pino-Wei \cite{DW} and Dancer-Wei \cite{DaW}.

 A  question naturally arises:
 \textit     {  if $1\le \ell \le N-3$  and $q=2^*_{N,\ell}$ or
  $q> 2^*_{N,\ell}$  (possibly different from a resonant sequence $q_j\nearrow+\infty $) does problem   \eqref{p1}   have one positive solutions and an arbitrary large number of sign changing solutions whose number increases as $\ep$ goes to zero?}

  In this paper   we give a positive answer  in  the pure supercritical case $q=2^*_{N,\ell}$ provided the domain $\mathcal D$ satisfies some symmetry assumptions.
In particular,   for any
   integer  $1\le\ell
 \leq N-3$ we build torus-like domains $\mathcal{D} $ and torus-like manifolds $\Gamma_\ell$ for which the number of sign-changing solutions to problem
\eqref{p1} with $q=2^*_{N,\ell}$ increases as $\ep$ goes to zero. These solutions have an arbitrary large number
of alternate positive and  negative layers which concentrate with different rates along the
$\ell$-dimensional  manifold $\Gamma_\ell . $
More precisely, let us state our main results.

Fix $\ell _{1},\ldots,\ell _{m}\in\mathbb{N}$ with $\ell :=\ell _{1}+\cdots+\ell _{m}\leq N-3$ and
a bounded smooth domain $\Omega$ in $\mathbb{R}^{n }$ with $n:=N-\ell$ such that
\begin{equation}
\overline{\Omega}\subset\{\left(  x_{1},\ldots,x_{m},x^{\prime}\right)
\in\mathbb{R}^{m}\times\mathbb{R}^{N-\ell -m}:x_{i}>0,\text{ }i=1,\ldots,m\}.
\label{omega}%
\end{equation}
Let $\xi_0\in\Omega$ be fixed and set $\Omega_\ep:=\Omega\setminus B(\xi_0,\ep) $ for $\ep$ small enough.
Set
\begin{equation}
\mathcal{D}:=\{(y^{1},\ldots,y^{m},z)\in\mathbb{R}^{\ell _{1}+1}\times\cdots
\times\mathbb{R}^{\ell _{m}+1}\times\mathbb{R}^{N-\ell -m}:\left(  \left\vert
y^{1}\right\vert ,\ldots,\left\vert y^{m}\right\vert ,z\right)  \in\Omega\}
\label{D}%
\end{equation}
and
\begin{equation}\label{gep}
\Gamma _\ell:=\{(y^{1},\ldots,y^{m},z)\in\mathbb{R}^{\ell_{1}+1}\times\cdots
\times\mathbb{R}^{\ell_{m}+1}\times\mathbb{R}^{N-\ell-m}:\left(  \left\vert
y^{1}\right\vert ,\ldots,\left\vert y^{m}\right\vert ,z\right)  =\xi_{0}\}.
\end{equation}
$\mathcal{D}$ is a bounded smooth domain in $\mathbb{R}^{N}$ and
 $\Gamma _\ell$ is an $\ell-$dimensional manifold in $\mathcal D$ which is diffeomorphic to $\mathbb{S}^{\ell_{1}}\times\cdots\times
\mathbb{S}^{\ell_{m}}.$ Here $\mathbb{S}^{d}$ is the unit sphere in
$\mathbb{R}^{d+1}.$
Set $\mathcal D_ \ep:=\mathcal   D\setminus \left\{x\in \mathcal D\ :\ \mathrm {dist}(x,\Gamma_\ell)\le \ep\right\}.$
Note that $\mathcal D,$ $\Gamma_\ell$ and $\mathcal D_\ep$ are
invariant under the action of the group $\Theta:=O(\ell _{1}+1)\times\cdots\times
O(\ell _{m}+1)$ on $\mathbb{R}^{N}$ given by%
\begin{equation*}
(g_{1},\ldots,g_{m})(y^{1},\ldots,y^{m},z):=(g_{1}y^{1},\ldots,g_{m}y^{m},z).
\end{equation*}
for every $g_{i}\in O(\ell _{i}+1),$ $y^{i}\in\mathbb{R}^{\ell _{i}+1},$
$z\in\mathbb{R}^{N-\ell -m}.$ Here, as usual, $O(d)$ denotes the group of linear
isometries of $\mathbb{R}^{d}.$ Let $q=2^*_{N,\ell} $ and note that
${2}_{N,\ell}^{\ast}=2^*_{n,0}.$  We shall
look for $\Theta$-invariant solutions to problem \eqref{p1}, i.e.
solutions $v$ of the form%
\begin{equation}
v(y^{1},\ldots,y^{m},z)=u(\left\vert y^{1}\right\vert ,\ldots,\left\vert
y^{m}\right\vert ,z). \label{inv}%
\end{equation}
A simple calculation shows that $v$ solves problem \eqref{p1} if and only
if $u$ solves
\begin{equation*}
-\Delta u-\sum_{i=1}^{m}\frac{\ell_{i}}{x_{i}}\frac{\partial u}{\partial x_{i}%
}=|u|^{2^*_{N,\ell}-2}u\quad\text{in}\ \Omega_\ep,\qquad u=0\quad\text{on}\ \partial\Omega_\ep.
\end{equation*}
This problem can be rewritten as%
\begin{equation}\label{p10}
-\text{div}(a(x)\nabla u)=a(x)|u|^{2^*_{n,0}-2}u\quad\text{in}\ \Omega_\ep,\qquad
u=0\quad\text{on}\ \partial\Omega_\ep,
\end{equation}
where $a(x_{1},\ldots,x_{n}):=x_{1}^{\ell_{1}}\cdots x_{m}^{\ell_{m}}.$
  Our
goal is to construct solutions $u_{\epsilon}$ to problem \eqref{p10} with an arbitrary large number of alternate positive and negative
bubbles which accumulate with different rates at the same point  $\xi_{0}$ as  $\epsilon\rightarrow0.$ They correspond, via (\ref{inv}),
to $\Theta$-invariant solutions $v_{\epsilon}$ of problem \eqref{p1} with
positive and negative layers which accumulate with different rates along the $\ell$-dimensional
 manifold $\Gamma_\ell $  defined in \eqref{gep}.

Thus, we are lead to study the more general  anisotropic critical problem
\begin{equation}\label{eq_main}
\left\{
\begin{aligned}
&-\text{div}(a(x)\nabla u) = a(x) |u|^{4 \over n-2}u\ &\hbox{in}\ \ome,\\
& u=0\ &\hbox{on}\ \partial\ome,\\
\end{aligned}\right.
\end{equation}
where $\Omega$ is a smooth bounded domain in $\rn$, $\xi_0\in\Omega$, $\ome := \Omega \setminus B(\xi_0, \ep)$ for $\ep$ small enough and  $a \in C^2(\overline{\Omega})$ satisfies $\min\limits_{x\in\Omega }a(x)>0.$

First of all, we construct sign-changing solutions of \eqref{eq_main} whose shape resemble a tower of bubbles with alternate sign centered at the point   $\xi_0$.
We recall that the bubbles
\begin{equation}\label{eq_udx}
\udx(x) := \alpha_n \({\delta \over {\delta^2 + |x -\xi|^2}}\)^{n-2 \over 2} \quad \text{for some } \delta > 0, ~ \xi \in \rn,
\end{equation}
where $\alpha_n = (n(n-2))^{n-2 \over 4}$ are the positive solution for the problem
\begin{equation}\label{eq_limit}
- \Delta u = u^{n+2 \over n-2} \quad \text{in } \rn, \quad u \in \mathcal D^{1,2}(\rn).
\end{equation}

Our main result concerning problem \eqref{eq_main} reads as follows.

 \begin{thm}\label{thm_main}
Suppose $n \ge 4$. For any integer $k \ge 1$, there is $\ep_k > 0$ such that for each $\ep \in (0, \ep_k)$, problem \eqref{eq_main} has a sign-changing solution $u_\ep$ which satisfies
\[u_\ep(x) =   \sik (-1)^i U_{\delta_i(\ep),\xi_i(\ep)}(x) + o(1) \quad \text{in}\ H^1(\Omega_\ep),\]
where  the weights $\delta_i(\ep)  $ and the centers $\xi_i(\ep)$  satisfy for any $i=1,\dots,k$
\begin{equation*}
\ep^{-{(n-2)+2(i-1) \over (n-1)+2(k-1)}} \delta_i(\ep) \to d_i > 0\ \hbox{and}\   \xi_i(\ep)\to\xi_0\quad \text{as}\ \ep \to 0.\end{equation*}
\end{thm}

It is clear that according to the previous discussion, by Theorem \ref{thm_main} we immediately deduce the following result concerning problem
 \eqref{p1}.

\begin{thm}
\label{main1} Assume $ 1\le\ell\le  N-   4.$ Then
for any integer $k$ there exists $\epsilon_{k}>0$ such that  for  any $\epsilon\in(0,\epsilon_{k}),$ problem
\eqref{p1} has a $\Theta$-invariant solution $v_{\epsilon}$ which
satisfies
\begin{equation*}
v_{\epsilon}(y)=\sum\limits_{i=1}^{k}(-1)^{i}\widetilde
{U}_{\delta_ i(\epsilon),\xi_ i(\epsilon)}(y)+o(1)\qquad\text{in
}H^1 (\mathcal{D}_\ep),
\end{equation*}
where  \begin{equation*}
\widetilde{U}_{\delta,\xi}(y^{1},\ldots,y^{m},z):=U_{\delta,\xi}(\left\vert
y^{1}\right\vert ,\ldots,\left\vert y^{m}\right\vert ,z),
\end{equation*}
the weights $\delta_i(\ep)  $ and the centers $\xi_i(\ep)$  satisfy
\begin{equation*}
\ep^{-{(n-2)+2(i-1) \over (n-1)+2(k-1)}} \delta_i(\ep) \to d_i > 0\ \hbox{and}\   \xi_i(\ep)\to\xi_0\quad \text{as}\ \ep \to 0.\end{equation*}
\ \end{thm}

 The proof of Theorem \ref{thm_main} relies on a very well known Lyapunov-Schmidt reduction.  In particular, we will follow the arguments
used in \cite{MP2,GMP}.
We shall omit many details of the proof
because they can be found, up to some minor modifications, in   \cite{MP2,GMP}.
We only compute what cannot be deduced from known results.
The paper is arranged as follows. Section  2 contains the main steps of the proof of Theorem \ref{thm_main}.
In Section 3 and in  Section 4 we study   the reduced energy. Appendix is devoted to prove some technical results which are
necessary to perform the finite dimensional reduction.

\section{The scheme of the proof of Theorem \ref{thm_main}}
Let   $\home$ be the Sobolev space  endowed with the inner product    $(v_1, v_2)_{\ep} := \inte a(\nabla v_1 \cdot \nabla v_2)$ for $v_1, v_2 \in \home$.
Also, denote $\|v\|_{\ep} = ((v, v)_{\ep})^{1 \over 2}$ for all $v \in \home$.
 Let  $i_{\ep}: \home \hookrightarrow L^{2n \over n-2}(\ome)$ be the Sobolev embedding
and  let  $i_{\ep}^*: L^{2n \over n+2}(\ome) \to \home$ be its adjoint so that $i_{\ep}^*(v) = u$ if and only if $-\text{div}(a\nabla u) = av$ in $\ome$ and $u = 0$ on $\pome$.
Note that there exists a constant $C > 0$ which depends only on the dimension $n$ such that $\|i_{\ep}^*(v)\|_{\ep} \le C\|v\|_{L^{2n \over n+2}(\ome)}$ for any $v \in L^{2n \over n+2}(\ome)$.

For any given $w \in D^{1,2}(\rn)$,
we denote by $P_{\ep}w \in \home$ the unique solution of the Dirichlet problem $\Delta P_{\ep} w = \Delta w$ in $\ome$ and $P_{\ep}w = 0$ on $\pome.$

\medskip
 The solutions we are looking for have the form
\begin{equation}\label{eq_ansatz}
u_{\ep} = V_{\ep} + \phi,\quad V_{\ep}:= \sik (-1)^{i+1} P_{\ep}U_i,
\end{equation}
where  the bubble $U_i := U_{\delta_i, \xi_i}$ is given in \eqref{eq_udx} with
\begin{equation}\label{eq_parameter}
\delta_i := \ep^{(n-2)+2(i-1) \over (n-1)+2(k-1)}d_i, \quad \xi_i := \xi_0 + \delta_i\sigma_i \quad \text{for some } d_i > 0,\ \sigma_i \in \rn \quad (i = 1, \cdots, k)
\end{equation}
and $\phi$ is a remainder term which belongs to a suitable space defined as follows.

\medskip
It is well known that the  space of solutions of the linearized equation
\begin{equation}\label{eq_linearized}
-\Delta \psi = p \udx^{p-1} \psi \quad \text{in } \rn,\ \|\psi\|_{L^{\infty}(\rn)} < \infty
\end{equation}
is spanned by $(n+1)$ functions
\begin{equation}\label{eq_pdxn}
\pdxn(x) := {\partial \udx \over \partial \delta} = \alpha_n \({n-2 \over 2}\) \delta^{n-4 \over 2} {|x - \xi|^2 -\delta^2 \over (\delta^2 + |x - \xi|^2)^{n \over 2}}
\end{equation}
and
\begin{equation}\label{eq_pdxj}
\pdxj(x) := {\partial \udx \over \partial \xi_j} = \alpha_n (n-2) \delta^{n-2 \over 2} {x_j - \xi_j \over (\delta^2 + |x - \xi|^2)^{n \over 2}}
\end{equation}
for $j = 1, \cdots, n$.

We set $\psi_i^j = \psi_{\delta_i, \xi_i}^j$ with parameters as in \eqref{eq_parameter} and we define the subspaces of $\home$
\[K_{\ep} = \text{span}\big\{P_{\ep}\psi_i^j : i = 1, \cdots, k, ~j = 0, 1, \cdots, n \big\}\]
and
\begin{equation}\label{eq_space_perp}
K_{\ep}\pe = \left\{\phi \in \home: \big(\phi, P_{\ep}\psi_i^j\big)_{\ep} = 0 \text{ for } i = 1, \cdots, k, ~j = 0, 1, \cdots, n \right\}.
\end{equation}
We also introduce the projections $\Pi_{\ep}: H_0^1(\ome) \to K_{\ep}$ and $\Pi_{\ep}\pe: H_0^1(\ome) \to K_{\ep}\pe$. As it is usual in the Lyapunov-Schmidt procedure, solving problem \eqref{eq_main} is   equivalent to finding a function $\phi \in K_{\ep}\pe$ and parameters $\md := (d_1, \cdots, d_k) \in (0,\infty)^k$ and $\ms := (\sigma_1, \cdots, \sigma_k) \in \(\rn\)^k$ which solve the couple of equations
\begin{equation}\label{eq_main_auxiliary}
\Pi_{\ep}\pe\(V_{\ep} + \phi - i_{\ep}^*\(|V_{\ep}+\phi|^{4 \over n-2}(V_{\ep}+\phi)\)\) = 0
\end{equation}
and
\begin{equation}\label{eq_main_finite}
\Pi_{\ep}\(V_{\ep} + \phi - i_{\ep}^*\(|V_{\ep}+\phi|^{4 \over n-2}(V_{\ep}+\phi)\)\) = 0
\end{equation}
where $V_{\ep}$ is a function given in \eqref{eq_ansatz} which depends on $\ep,$ $\md$ and $\ms$.

\medskip
Firstly we solve equation \eqref{eq_main_auxiliary}.
  \begin{prop}\label{prop_auxiliary}
For any compact set $\mathcal C\subset (0,\infty)^k\times(\mathbb R^n)^k,$ there exists $\ep_0 > 0$ such that for each $\ep \in (0, \ep_0)$ and $(\md, \ms) \in \mathcal C$, equation \eqref{eq_main_auxiliary} possesses a unique solution $\phi_{\ep}^{\md, \ms} \in K_{\ep}$ satisfying
\begin{equation}\label{eq_auxiliary_1}
\left\|\phi_{\ep}^{\md, \ms}\right\|^2 = o\(\ep^{n-2 \over (n-1)+2(k-1)}\).
\end{equation}
Moreover, $({\md, \ms}) \mapsto \phi_{\ep}^{\md, \ms}$ is a $C^1$-map.
\end{prop}
\noindent
The proof is postponed in   Appendix \ref{sec_sketch}.

\medskip
Then, the problem reduces to find $(\md, \ms)$ which solves \eqref{eq_main_finite}. Notice that equation \eqref{eq_main} has a variational structure, namely, its solutions are critical points of the energy functional
\[I_{\ep}(u) = {1 \over 2}\inte a(x)|\nabla u|^2dx - {1 \over p+1} \inte a(x)|u|^{p+1}dx \quad \text{for } u \in \home.\]
We introduce the reduced energy
\[J_{\ep}(\md, \ms) := I_{\ep}\(V_{\ep}^{\md, \ms}  + \phi_{\ep}^{\md, \ms}\) \quad \text{for } (\md, \ms) \in (0,\infty)^k \times \(\rn\)^k\]
where the superscript of $V_{\ep}^{\md, \ms} := V_{\ep}$ (see \eqref{eq_ansatz}) emphasizes its dependence on $\md$ and $\ms$.
Next, we reduce the problem to  a finite dimensional one.
 \begin{prop}\label{prop_reduction}
If $(\md, \ms)$ is a critical point of $J_{\ep}$, then $V_{\ep}^{\md, \ms}  + \phi_{\ep}^{\md, \ms}$ is a solution of \eqref{eq_main}.
\end{prop}
\noindent The proof is postponed in  Appendix \ref{sec_sketch}.

\medskip
Therefore, since the problem is reduced to looking for a critical point of the reduced energy $J_{\ep}$,
we need its asymptotic expansion.
\begin{prop}\label{prop_expansion}
Assume $n \ge 4$. It holds true that
 \begin{equation}\label{eq_expansion}
J_{\ep}(\md, \ms) = c_1k a(\xi_0) + \Psi(\md, \ms) \ep^{n-2 \over (n-1)+2(k-1)} + o\(\ep^{n-2 \over (n-1)+2(k-1)}\)
\end{equation}
$C^1$-uniformly in compact sets of $(0,\infty)^k \times \(\rn\)^k.$ Here
\begin{equation}\label{eq_psi_ep}
\Psi(\md, \ms) = c_2 \la \nabla a(\xi_0), \sigma_1 \ra d_1 + {c_3 a(\xi_0) \over (1+|\sigma_k|^2)^{n-2}} {1 \over d_k^{n-2}} + \sum\limits_{i=1}^{k-1} {c_4 a(\xi_0) \over (1+|\sigma_i|^2)^{n-2 \over 2}} \({d_{i+1} \over d_i}\)^{n-2 \over 2}
\end{equation}
and $c_1,\dots,c_4$ are positive constants.
\end{prop}
\noindent The proof is postponed in  Section \ref{sec_expansion}.

The last step consists in showing that the leading term of the reduced energy has a critical point which is stable under $C^1-$perturbation.
 \begin{prop}\label{prop_critical}
$\Psi$ has a nondegenerated critical point.
\end{prop}
\noindent The proof is postponed in   Section \ref{sec_critical}.

Finally, we have all the ingredients to complete the proof of Theorem  \ref{thm_main}.

\begin{proof}[Proof of Theorem \ref{thm_main}]
By Proposition \ref{prop_critical} and Proposition \ref{prop_expansion} it follows that $J_{\ep}$ has a critical point provided    $\ep  $ is small enough.
By Proposition \ref{prop_reduction} the claim immediately follows.
\end{proof}
\section{Expansion of the energy of approximate solutions}\label{sec_expansion}
The purpose of this section is to provide the proof of Proposition \ref{prop_expansion}.
We start this section by recalling the lemma \cite[Lemma 3.1]{GMP}, which describes the difference between $\udx$ and its projection $P_\ep \udx$ onto $H_0^1(\ome)$. Denote by $G(x,y)$ the Green function associated to $-\Delta$ with Dirichlet boundary condition and $H(x,y)$ its regular part, that is, let $G(x,y)$ and $H(x,y)$ be functions defined by
\[\left\{{\setlength\arraycolsep{2pt}\begin{array}{rll}
-\Delta_x G(x,y) &= \delta_y (x) &\quad \text{for } x \in \Omega,\\
G(x,y) &= 0 & \quad \text{for } x \in \partial\Omega,
\end{array}} \right.\]
and
\[G(x,y) = \gamma_n \( \frac{1}{|x-y|^{n-2}} - H(x,y) \) \quad \text{ where} \quad \gamma_n = \frac{1}{(n-2)|S^{n-1}|}\]
where $|S^{n-1}| = (2\pi^{n/2})/ \Gamma(n/2)$ is the Lebesgue measure of the $(n-1)$-dimensional unit sphere.

\begin{lemma}\label{lemma_expansion_1}
Given $\delta > 0$ and $\xi \in \ome$ such that $\ep = o(\delta)$ as $\ep \to 0$ and $|\xi - \xi_0| \le c \delta$ for some $c > 0$ fixed, the following expansions are valid.
\begin{align*}
\udx(x) &= P_\ep \udx(x) + \alpha_n \delta^{n-2 \over 2} H(x,\xi) +
\alpha_n{\delta^{n-2 \over 2} \over (\delta^2+|\xi-\xi_0|^2)^{n-2 \over 2}}{\ep^{n-2} \over |x-\xi_0|^{n-2}}\\
&\hspace{130pt} +\delta^{n-2 \over 2} \cdot O\(\left\{\ep^{n-2} + \bigg({\ep \over \delta}\bigg)^{n-1}\right\}{1 \over |x-\xi_0|^{n-2}} + \delta^2 + \bigg({\ep \over \delta}\bigg)^{n-2}\),\\
\pdxn(x) &= P_\ep \pdxn(x) + \alpha_n \({n-2 \over 2}\) \delta^{n-4 \over 2} H(x,\xi) +
\alpha_n\({n-2 \over 2}\) \delta^{n-4 \over 2} {|\xi-\xi_0|^2 - \delta^2 \over (\delta^2+|\xi-\xi_0|^2)^{n \over 2}}{\ep^{n-2} \over |x-\xi_0|^{n-2}}\\
&\hspace{130pt} +\delta^{n-4 \over 2} \cdot O\(\left\{\ep^{n-2} + \bigg({\ep \over \delta}\bigg)^{n-1}\right\}{1 \over |x-\xi_0|^{n-2}} + \delta^2 + \bigg({\ep \over \delta}\bigg)^{n-2}\)
\intertext{and}
\pdxj(x) &= P_\ep \pdxj(x) + \alpha_n \delta^{n-2 \over 2} {\partial H \over \partial \xi_j} (x,\xi) -
\alpha_n (n-2) \delta^{n-2 \over 2} {(\xi-\xi_0)_j \over (\delta^2+|\xi-\xi_0|^2)^{n \over 2}}{\ep^{n-2} \over |x-\xi_0|^{n-2}}\\
&\hspace{160pt} +\delta^{n-2 \over 2} \cdot O\(\left\{\ep^{n-2} + {\ep^{n-1} \over \delta^n}\right\}{1 \over |x-\xi_0|^{n-2}} + \delta^2 + {\ep^{n-2} \over \delta^{n-1}}\)
\end{align*}
for any $x \in \ome$ and $j = 1, \cdots, n$. In particular,
\begin{equation}\label{eq_expansion_1_1}
\big|\pdxn - P_\ep \pdxn\big| = O\(\delta^{n-4 \over 2} + {\ep^{n-2} \over \delta^{n \over 2} |x-\xi_0|^{n-2}}\), \quad \big|\pdxj - P_\ep \pdxj\big| = O\(\delta^{n-2 \over 2} + {\ep^{n-2} \over \delta^{n \over 2} |x-\xi_0|^{n-2}}\)
\end{equation}
for $j = 1, \cdots, n$ and
\[|\udx - P_\ep \udx| = O\(\delta^{n-2 \over 2} + {\ep^{n-2} \over \delta^{n-2 \over 2} |x-\xi_0|^{n-2}}\), \quad x \in \ome.\]
\end{lemma}

Let us choose $\rho > 0$ so small that $\overline{B(\xi_0, \rho)} \subset \Omega$. Also, introduce the annulus $A_l$ whose definition is given by
\begin{equation}\label{eq_annulus}
A_l = B\(\xi_0, \sqrt{\delta_{l-1}\delta_l}\) \setminus B\(\xi_0, \sqrt{\delta_l\delta_{l+1}}\) \quad \text{with } \delta_0 = {\rho^2 \over \delta_1} \text{ and } \delta_{k+1} = {\ep^2 \over \delta_k}
\end{equation}
for $l = 1, \cdots, k$. Then clearly
\[\ome = \(\bigcup_{l=1}^k A_l\) \cup \(\ome \setminus B(\xi_0, \rho)\).\]
As a consequence of Lemma \ref{lemma_expansion_1}, we can obtain
\begin{lemma}\label{lemma_expansion_2}
For any $i, l = 1, \cdots, k$, $i \ne l$, the following estimations are satisfied.
\begin{align*}
&(i) \intal aU_l^{2n \over n-2} = a(\xi_0) \intal U_l^{2n \over n-2} + \intal (a-a(\xi_0)) U_l^{2n \over n-2}\\
&\ \quad = (n(n-2))^{n \over 2} \(a(\xi_0) + \delta_{l1} \la \nabla a(\xi_0), \sigma_1 \ra d_1\ep^{n-2 \over (n-1)+2(k-1)}\) \int_{\rn} {dy \over (1 + |y|^2)^n} + o\(\ep^{n-2 \over (n-1)+2(k-1)}\);\\
&(ii) \intal U_l^{n+2 \over n-2}(PU_l - U_l) \\
&\ \quad = - (n(n-2))^{n \over 2} |B_n| \cdot {\delta_{lk} \over (1+|\sigma_k|^2)^{n-2}} \ep^{n-2 \over (n-1)+2(k-1)} + o\(\ep^{n-2 \over (n-1)+2(k-1)}\) = O\bigg(\delta_l^{n-2} + \({\ep \over \delta_l}\)^{n-2}\bigg);\\
&(iii) \intal U_l^{n+2 \over n-2}U_i\\
&\ \quad = (n(n-2))^{n \over 2} |B_n| \cdot \left[{\delta_{i(l+1)} \over (1+|\sigma_l|^2)^{n-2 \over 2}} \({d_{l+1} \over d_l}\)^{n-2 \over 2} + {\delta_{i(l-1)}  \over (1+|\sigma_{l-1}|^2)^{n-2 \over 2}} \({d_{l+1} \over d_l}\)^{n-2 \over 2} \right] \ep^{n-2 \over (n-1)+2(k-1)} \\
&\quad \quad + o\(\ep^{n-2 \over (n-1)+2(k-1)}\);\\
&(iv) \intal |PU_l-U_l|^{2n \over n-2},\ \intal U_i^{2n \over n-2} = O\(\ep^{n \over (n-1)+2(k-1)}\);\\
&(v) \int_{\ome \setminus A_l} a U_l^{p+1},\ \inte a U_l^p (PU_i -U_i),\ \int_{\ome \setminus A_l} a U_l^p U_i = o\(\ep^{n-2 \over (n-1)+2(k-1)}\);\\
&(vi) \inte (a - a(\xi_0)) U_l^{n+2 \over n-2}(PU_l - U_l),\ \inte (a - a(\xi_0))U_l^{n+2 \over n-2}U_i = o\(\ep^{n-2 \over (n-1)+2(k-1)}\)
\end{align*}
where $\delta_{lk}$ is the Kronecker delta and $|B_n| = \pi^{n/2}/\Gamma(n/2+1)$ denotes the volume of the $n$-dimensional unit ball.
\end{lemma}
\noindent and
\begin{lemma}\label{lemma_expansion_3}
Assume $i,\ l = 1, \cdots, k$. If $n \ge 4$, we have
\begin{equation}\label{eq_expansion_3_1}
\inte (\nabla a \cdot \nabla P_\ep U_i) P_\ep U_l = o\(\ep^{n-2 \over (n-1)+2(k-1)}\)
\end{equation}
and
\begin{multline}\label{eq_expansion_3_5}
\inte (\nabla a \cdot \nabla P_\ep U_i) \cdot \(\ep^{(n-2) + 2(l-1) \over (n-1)+2(k-1)} P_\ep \psi_l^j\)\\
= \begin{cases}
\delta_{i1}\delta_{l1} \dfrac{\partial a}{\partial x_j}(\xi_0) c_2 \ep^{n-2 \over (n-1)+2(k-1)} + o\(\ep^{n-2 \over (n-1)+2(k-1)}\) &\text{for } j = 1, \cdots, n\\
o\(\ep^{n-2 \over (n-1)+2(k-1)}\) &\text{for } j = 0.
\end{cases}
\end{multline}
\end{lemma}
\begin{rmk}
In \cite[Lemma A.2]{ACP} (see also \cite[Lemma A.9]{KP2} and \eqref{eq_ortho_1} below), the authors proved that
\[\inte (\nabla a \cdot \nabla P_\ep U_i) P_\ep U_l,\quad \inte (\nabla a \cdot \nabla P_\ep U_i) \cdot \big(\delta_l P_\ep \psi_l^j\big) = O\(\delta_1^{1-\eta}\)
\quad \text{for some small } \eta > 0\]
by utilizing Young's inequality.
However, this estimation is insufficient in our situation so we will pursue another approach making the use of Lemma \ref{lemma_expansion_1} in a direct way, which turns out to be more complicated.
\end{rmk}
\begin{proof}[Proof of Lemma \ref{lemma_expansion_2}]
The computations follow as an application of Lemma \ref{lemma_expansion_1} or by the direct computation using the definition of $U_l$ in \eqref{eq_udx}.
For the detailed exposition in similar settings, see \cite[Section 3]{GMP} and \cite[Lemma 4.2]{KP2}.
\end{proof}
\begin{proof}[Proof of Lemma \ref{lemma_expansion_3}]
We prove \eqref{eq_expansion_3_1} first. To do this, we decompose the left-hand side of \eqref{eq_expansion_3_1} into
\begin{equation}\label{eq_expansion_3_2}
\inte (\nabla a \cdot \nabla P_\ep U_i) P_\ep U_l = \inte \left[\nabla a \cdot \nabla (P_\ep U_i - U_i)\right] P_\ep U_l + \inte \(\nabla a \cdot \nabla U_i\) (P_\ep U_l-U_l) + \inte (\nabla a \cdot \nabla U_i) U_l
\end{equation}
and estimate each of terms in the right-hand side. For brevity, we will use $PU_i = P_\ep U_i$.

\medskip
First, we compute the first term. Since $0 \le P U_i \le U_i$ in $\ome$ for all $i = 1, \cdots, k$, it holds
\begin{align*}
\inte |\nabla a| \cdot |\nabla (PU_i - U_i)| P U_l &\le C \|P U_i - U_i\|_{H^1(\ome)}\cdot\|U_l\|_{L^2(\ome)} \\
&\le \begin{cases}
C \delta_l \cdot \|P U_i - U_i\|_{H^1(\ome)} & \text{if } n \ge 5,\\
C \delta_l|\log\delta_l| \cdot \|P U_i - U_i\|_{H^1(\ome)} & \text{if } n = 4.
\end{cases}
\end{align*}
Moreover, since
\[\inte |\nabla PU_i|^2 = \inte \nabla U_i \cdot \nabla PU_i = \inte U_i^{n+2 \over n-2} PU_i,\]
by applying \eqref{eq_pdxj}, we obtain that
\begin{align*}
&\ \|PU_i - U_i\|_{H^1(\ome)}^2\\
&= - \alpha_n^{2n \over n-2} \inte {\delta_i^n \over (\delta_i^2 + |x-\xi_i|^2)^n} + \alpha_n^2 (n-2)^2 \delta_i^{n-2} \inte {|x-\xi_i|^2 \over (\delta_i^2 + |x-\xi_i|^2)^n} + \inte U_i^{n+2 \over n-2} (U_i - PU_i)\\
&= - \left[\alpha_n^{2n \over n-2} \int_{\rn} {1 \over (1 + |y|^2)^n} + O\(\delta_i^n + \({\ep \over \delta_i}\)^n\)\right] \\
&\quad  + \left[\alpha_n^2 (n-2)^2 \int_{\rn} {|y|^2 \over (1 + |y|^2)^n} + O\(\delta_i^{n-2} + \({\ep \over \delta_i}\)^n\)\right] + O\bigg(\delta_i^{n-2} + \({\ep \over \delta_i}\)^{n-2}\bigg),
\end{align*}
namely,
\[\|PU_i - U_i\|_{H^1(\ome)} = O\(\delta_i^{n-2 \over 2} + \({\ep \over \delta_i}\)^{n-2 \over 2}\).\]
Thus the first term is $o(\delta_1)$ provided $n \ge 4$. Note that our argument gives only $\inte [\nabla a \cdot \nabla (P_\ep U_i - U_i)] P_\ep U_l = O(\sqrt{\delta_1}) \ne o(\delta_1)$ if $n = 3$.

\medskip
Next, we consider the second term. From Lemma \ref{lemma_expansion_1}, we deduce
\[\inte |\nabla a||\nabla U_i||PU_l-U_l| \le C \inte |\nabla U_i|\(\delta_l^{n-2 \over 2} + {\ep^{n-2} \over \delta_l^{n-2 \over 2} |x-\xi_0|^{n-2}}\).\]
Denoting $\widetilde{\Omega}_\ep = (\ome - \xi_0)/\delta_i$, we get from \eqref{eq_pdxj} that
\[\delta_l^{n-2 \over 2} \inte |\nabla U_i| dx = \delta_l^{n-2 \over 2}\delta_i^{n \over 2} \intwe {|y - \sigma_i| \over (1+ |y - \sigma_i|^2)^{n \over 2}} dy = O\((\delta_i\delta_l)^{n-2 \over 2}\) = o(\delta_1).\]
Moreover there holds
\[{\ep^{n-2} \over \delta_l^{n-2 \over 2}} \cdot \inte |\nabla U_i| \cdot {1 \over |x-\xi_0|^{n-2}}dx = C \delta_i \cdot \({\ep \over \delta_i}\)^{n-2 \over 2}\({\ep \over \delta_l}\)^{n-2 \over 2} \intwe {|y - \sigma_i| \over (1 + |y - \sigma_i|^2)^{n \over 2}} \cdot {1 \over |y|^{n-2}} dy = o(\delta_1),\]
from which we conclude that the second term is $o(\delta_1)$ .

\medskip
As a result, it suffices to show that the third term is also $o(\delta_1)$. We consider when $i = l$ first.
To estimate the term for this case, we will divide the domain $\ome$ into two disjoint sets $B(\xi_0, \sqrt{\delta_i}) \setminus B(\xi_0, \ep)$ and $\Omega \setminus B(\xi_0, \sqrt{\delta_i})$
and then deal with each of the integrations of $(\nabla a \cdot \nabla U_i)U_i$ over these domains.
Employing the dimension assumption $n \ge 4$, we find that
\begin{align*}
\int_{\Omega \setminus B(\xi_0, \sqrt{\delta_i})} |\nabla a \cdot \nabla U_i| U_i &\le C \int_{\rn \setminus B(\xi_0, \sqrt{\delta_i})} {\delta_i^{n-2} |x - \xi_i| \over (\delta_i^2 + |x - \xi_i|^2)^{n-1}} dx \\
& = C \delta_i \int_{\rn \setminus B(0, 1/\sqrt{\delta_i})} {|y - \sigma_i| \over (1 + |y - \sigma_i|^2)^{n-1}} dy = o(\delta_i).
\end{align*}
On the other hand, we see
\begin{equation}\label{eq_expansion_3_6}
\begin{aligned}
&\ \left| \int_{B(\xi_0, \sqrt{\delta_i}) \setminus B(\xi_0, \ep)} (\nabla a \cdot \nabla U_i) U_i\right|\\
&= \sum_{j=1}^n \left| {\partial a \over \partial x_j}(\xi_0)\right| \cdot \left|\int_{B(\xi_0, \sqrt{\delta_i}) \setminus B(\xi_0, \ep)} {\partial U_i \over \partial x_j}\cdot U_i \right|
+ \left| \int_{B(\xi_0, \sqrt{\delta_i}) \setminus B(\xi_0, \ep)} \left\{(\nabla a - \nabla a(\xi_0)) \cdot \nabla U_i\right\} U_i \right| \\
&\le C\delta_i \sum_{j=1}^n \left|{\partial a \over \partial x_j}(\xi_0)\right| \cdot \left| \int_{B(0, 1/\sqrt{\delta_i}) \setminus B(0, \ep/\delta_i)} {(y - \sigma_i)_j \over (1 + |y - \sigma_i|^2)^{n-1}}dy \right| \\
&\hspace{130pt} + \delta_i \int_{B(0, 1/\sqrt{\delta_i}) \setminus B(0, \ep/\delta_i)} \left|\nabla a(\xi_0 + \delta_i y) - \nabla a (\xi_0)\right| \cdot {|y - \sigma_i| \over (1 + |y - \sigma_i|^2)^{n-1}}dy.
\end{aligned}
\end{equation}
Since
\[\int_{B(0, 1/\sqrt{\delta_i}) \setminus B(0, \ep/\delta_i)} {(y - \sigma_i)_j \over (1 + |y - \sigma_i|^2)^{n-1}}dy = - \(\int_{B(0, \ep/\delta_i)} + \int_{B(0, 1/\sqrt{\delta_i})^c}\) {(y - \sigma_i)_j \over (1 + |y - \sigma_i|^2)^{n-1}}dy = o(1)\]
and
\[\left|\nabla a(\xi_0 + \delta_i y) - \nabla a (\xi_0)\right| \le \|D^2a\|_{L^{\infty}(\Omega)} \delta_i |y| \le C \sqrt{\delta_i} \quad \text{for any } y \in B(0, 1/\sqrt{\delta_i}),\]
we arrive at $\int_{B(\xi_0, \sqrt{\delta_i}) \setminus B(\xi_0, \ep)} (\nabla a \cdot \nabla U_i) U_i = o(\delta_i)$. This implies that the third term of the right-hand side in \eqref{eq_expansion_3_2} is $o(\delta_i)$ if $i = l$. For $i \ne l$, we have
\begin{multline*}
\inte |\nabla U_i|U_l = C \inte {\delta_i^{n-2 \over 2} |x-\xi_i| \over (\delta_i^2 + |x - \xi_i|^2)^{n \over 2}}{\delta_l^{n-2 \over 2} \over (\delta_l^2 + |x - \xi_l|^2)^{n-2 \over 2}} dx\\
\le \begin{cases}
\begin{aligned}
C \delta_i \({\delta_l \over \delta_i}\)^{n-2 \over 2} \int_{\rn} {|y - \sigma_i| \over (1 + |y-\sigma_i|^2)^{n \over 2}} \cdot \dfrac{1}{|y - (\delta_l/\delta_i)\sigma_l|^{n-2}} dy = o(\delta_i) = o(\delta_1)
\end{aligned} & \text{if } i < l,\\
\begin{aligned}
&\({\delta_i \over \delta_1}\)^{n-2 \over 2} \int_{B(\xi_0, \sqrt{\delta_1}) \setminus B(\xi_0, \ep)} {|x-\xi_i|\ dx \over (\delta_i^2 + |x - \xi_i|^2)^{n \over 2}} + \|\nabla U_i\|_{L^2(\Omega \setminus B(\xi_0, \sqrt{\delta_1}))} \|U_1\|_{L^2(\ome)}\\
&= O(\delta_1) \int_{B\(0, {\sqrt{\delta_1} \over \delta_i}\) \setminus B\(0, {\ep \over \delta_i}\)} {\delta_i |y-\sigma_i|\ dy \over (1 + |y - \sigma_i|^2)^{n \over 2}} + O(\delta_1|\log\delta_1|)\cdot O\({\delta_i \over \sqrt{\delta_1}}\)^{n-2 \over 2}\\
&= O(\delta_1) O\(\sqrt{\delta_1}\) + o(\delta_1) = o(\delta_1)
\end{aligned} &\text{if } i > l = 1,\\
C \delta_i^{-1} \|U_i\|_{L^2(\ome)} \|U_l\|_{L^2(\ome)} =  O\(|\log\delta_i| \cdot \delta_l |\log\delta_l|\) = o(\delta_1) & \text{if } i > l \ge 2.
\end{cases}
\end{multline*}
Hence \eqref{eq_expansion_3_1} is true.

\medskip
The derivation of \eqref{eq_expansion_3_5} goes along the same way as the above except the part that corresponds to \eqref{eq_expansion_3_6}.
In this case, instead of \eqref{eq_expansion_3_6}, we have that
\begin{align*}
\int_{B(\xi_0, \sqrt{\delta_i}) \setminus B(\xi_0, \ep)} (\nabla a \cdot \nabla U_i) \big(\delta_i \psi_i^j\big)
&= \sum_{m=1}^n {\partial a \over \partial x_m}(\xi_0) \int_{B(\xi_0, \sqrt{\delta_i}) \setminus B(\xi_0, \ep)} {\partial U_i \over \partial x_m}\cdot \big(\delta_i \psi_i^j\big) + o(\delta_1)\\
&= \delta_{i1} {\partial a \over \partial x_j}(\xi_0) c_2 \delta_1 + o(\delta_1)
\end{align*}
for $i = 1, \cdots, k$ where $c_2$ is a constant defined in \eqref{eq_expansion_3_4} below.
Thus \eqref{eq_expansion_3_5} follows and the proof of Lemma \ref{lemma_expansion_3} is completed.
\end{proof}

Using Proposition \ref{prop_auxiliary}, we can also check that
\begin{lemma} It holds that
\[I_{\ep}\big(V_{\ep}^{\md, \ms}  + \phi_{\ep}^{\md, \ms}\big) - I_{\ep}\big(V_{\ep}^{\md, \ms}\big) = o\(\ep^{n-2 \over (n-1)+2(k-1)}\).\]
\end{lemma}
\begin{proof}
We refer to the proof of \cite[Lemma 4.1]{KP2}.
\end{proof}

Now we are ready to prove that \eqref{eq_expansion} holds $C^1$-uniformly on compact sets of $(0,\infty)^k\times(\mathbb R^n)^k$.
\begin{proof}[Proof of Proposition \ref{prop_expansion}]
By the previous lemma, it is sufficient to show that
\[I_{\ep}\big(V_{\ep}^{\md, \ms}\big) = c_1k + \Psi(\md, \ms) \ep^{n-2 \over (n-1)+2(k-1)} + o\(\ep^{n-2 \over (n-1)+2(k-1)}\)\]
where $\Psi$ is the function given in \eqref{eq_psi_ep} and $c_1 > 0$ is a fixed quantity. For simplicity, we set $p = (n+2)/(n-2)$ and omit the subscripts and superscripts of $V_{\ep}^{\md, \ms}$.

\medskip
As in (4.6) and (4.7) in \cite{KP2}, we write
\begin{multline*}
{1 \over 2} \inte a|\nabla V|^2 = {1 \over 2} \slk \left[\intal a U_l^{p+1} + \inte a U_l^p (PU_l -U_l) + \int_{\ome \setminus A_l} a U_l^{p+1} - \inte (\nabla a \cdot \nabla PU_l) PU_l \right]\\
+ \sum_{l < i} (-1)^{l+i} \left[\intal a U_l^pU_i + \inte a U_l^p (PU_i -U_i) + \int_{\ome \setminus A_l} a U_l^p U_i - \inte (\nabla a \cdot \nabla PU_l) PU_i \right]
\end{multline*}
and
\begin{equation}\label{eq_expansion_3_3}
\begin{aligned}
&\ {1 \over p + 1} \inte a|V|^{p+1} dx \\
&= \slk \bigg[ {1 \over p + 1} \intal a U_l^{p+1} + \intal aU_l^p(PU_l - U_l) \bigg]
+ \sum_{i \ne l} (-1)^{i+l} \bigg[ \intal aU_l^p U_i + \intal aU_l^p (PU_i-U_i) \bigg]\\
&\ + p\int_0^1 (1-\theta) \intal a \bigg| (-1)^{l+1} U_l + \theta\Big[ (-1)^{l+1}(PU_l-U_l) + \sum_{i \ne l} (-1)^{i+1} PU_i\Big]\bigg|^{p-1}\\
&\hspace{190pt} \bigg((-1)^{l+1}(PU_l-U_l) + \sum_{i \ne l} (-1)^{i+1} PU_i\bigg)^2 dx d\theta + O\(\delta_1^n\).
\end{aligned}
\end{equation}
Following the computations which were conducted to obtain (4.8) in \cite{KP2}, we see that
\[\intal aU_l^pU_i = \int_{A_i} aU_i^pU_l + o(\delta_1) \quad \text{for any pair } (i, l) \text{ such that } i < l.\]
Furthermore, (ii)-(iv) in Lemma \ref{lemma_expansion_2} lead us to observe that the last term $p \int_0^1(1-\theta)\intal \cdots dxd\theta$ in the right-hand side in \eqref{eq_expansion_3_3} is $o(\delta_1)$ (see (4.14) in \cite{KP2}). As a result, (v) in Lemma \ref{lemma_expansion_2} and Lemma \ref{lemma_expansion_3} yield
\begin{multline*}
{1 \over 2} \inte a|\nabla V|^2 - {1 \over p + 1} \inte a|V|^{p+1} dx \\
= {1 \over n} \slk \intal a U_l^{p+1} - {1 \over 2} \slk \inte a U_l^p (PU_l -U_l) - \sum_{l < i} (-1)^{l+i} \intal a U_l^pU_i+ o(\delta_1).
\end{multline*}
However, (i)-(iii) and (vi) in Lemma \ref{lemma_expansion_2} then give us  \eqref{eq_expansion} and \eqref{eq_psi_ep} with
\begin{equation}\label{eq_expansion_3_4}
c_1 = c_2 = {1 \over n} \cdot (n(n-2))^{n \over 2} \int_{\rn} {dy \over (1+|y|^2)^n} \quad \text{and} \quad 2c_3 = c_4 = (n(n-2))^{n \over 2}|B_n|,
\end{equation}
as we desired.

\medskip
Having \eqref{eq_expansion_3_5} in mind, we can perform the $C^1$-estimate in a similar way to \cite[Subsection 5.1]{KP2}, which we omit.
\end{proof}

\section{Existence of critical points for $\Psi$}\label{sec_critical}
Here we will give the proof of Proposition \ref{prop_critical}. If $k = 1$ or $2$, then the proof is relatively simple so we will assume here $k \ge 3$.

\medskip
Setting $d_1 = t_1,\ d_2 = t_1t_2, \cdots,\ d_k = t_1t_2\cdots t_k$ and writing $\mt = (t_1, \cdots, t_k)$,
let us define
\[\wpe(\mt, \ms) = \Psi(\md(\mt), \ms) = c_2 \la \nabla a(\xi_0), \sigma_1 \ra t_1 + \sum\limits_{i=1}^{k-1} {c_4 a(\xi_0) \over (1+|\sigma_i|^2)^{n-2 \over 2}} t_{i+1}^{n-2 \over 2} + {c_3 a(\xi_0) \over (1+|\sigma_k|^2)^{n-2}} {1 \over (t_1\cdots t_k)^{n-2}}\]
for any
\[(\mt, \ms) \in (0,\infty)^k \times \Xi \quad \text{where } \Xi := \left\{(\sigma_1, \cdots, \sigma_k) : \la \nabla a(\xi_0), \sigma_1 \ra > 0, \sigma_2, \cdots, \sigma_k \in \rn\right\}.\]
Then for each fixed $\ms \in \Xi$, there exists the unique point $\mt = \mt(\ms)$ such that $\dfrac{\partial \wpe}{\partial t_i}(\mt, \ms) = 0$ ($i = 1, \cdots, k$).
In fact, after some computations, one can show that $\dfrac{\partial \wpe}{\partial t_1}(\mt, \ms) = 0$ and $\dfrac{\partial \wpe}{\partial t_i}(\mt, \ms) = 0$ ($i= 2, \cdots, k$) are equivalent to
\begin{equation}\label{eq_critical_1}
c_2 \la \nabla a(\xi_0), \sigma_1 \ra t_1 = {c_4(n-2)a(\xi_0) \over 2(1+|\sigma_{i-1}|)^{n-2 \over 2}} \cdot t_i^{n-2 \over 2} = {c_3(n-2)a(\xi_0) \over (1 + |\sigma_k|^2)^{n-2}} \cdot {1 \over (t_1 \cdots t_k)^{n-2}}.
\end{equation}
This system is uniquely solvable and the solution is given by
\[t_1 = {c_4(n-2)a(\xi_0) \over 2 c_2 \la \nabla a(\xi_0), \sigma_1 \ra (1 + |\sigma_1|^2)^{n-2 \over 2}} \cdot t_2^{n-2 \over 2}, \quad t_i = {1 + |\sigma_i|^2 \over 1 + |\sigma_1|^2} \cdot t_2 \ (i = 3, \cdots, k)\]
and
\[t_2 = \left[{c_2 \la \nabla a(\xi_0), \sigma_1 \ra \over c_3(n-2)a(\xi_0)} \cdot (1+|\sigma_1|^2)^{n+2k-5 \over 2}\cdot \prod_{i=2}^k {1 \over 1+|\sigma_i|^2} \right]^{2 \over n+2k-3}.\]
Also, \eqref{eq_critical_1} and the relation $2c_3 = c_4$ (see \eqref{eq_expansion_3_4}) ensure that
\begin{align*}
\widehat{\Psi}(\ms) &:= \wpe(\mt(\ms), \ms)\\
&= {n+2k-3 \over n-2} \left[c_2^{n-2} (c_3a(\xi_0)(n-2))^{2k-1} \right]^{1 \over n+2k-3} \cdot \left[{\la \nabla a(\xi_0), \sigma_1 \ra \over 1+|\sigma_1|^2} \cdot \prod_{i=2}^k {1 \over 1+|\sigma_i|^2}\right]^{n-2 \over n+2k-3}.
\end{align*}
By inspection, we can see that $\widehat{\Psi}$ has a maximum point $\ms_0 = \(\nabla a(\xi_0)/|\nabla a(\xi_0)|, 0, \cdots, 0\)$ in $\Xi$. Hence $(\mt(\ms_0), \ms_0)$ is a critical point of $\wpe$.

\medskip
We claim that $(\mt(\ms_0), \ms_0)$ is a nondegenerate critical point of $\wpe$. Without loss of generality, we may assume that $\nabla a(\xi_0) = (|\nabla a(\xi_0)|, 0, \cdots, 0)$. Then the determinant of the Hessian matrix  $D^2\wpe(\mt(\ms_0), \ms_0)$ is
\[\det \(D^2\wpe(\mt(\ms_0), \ms_0)\) = C \det\(
\begin{array}{cc}
\mca & \mcb \\
\mcb^t & \mcd
\end{array}\)\]
where $C > 0$ is some constant and $\mca := \mca_1 + \mca_2,\ \mcb$ and $\mcd$ are $k \times k$, $k \times n$ and $n \times n$ matrices, respectively, defined by
\begin{align*}
\mca_1 &= b_2f(\mt) \(
\begin{array}{ccccc}
{n-1 \over t_1^2} & {n-2 \over t_1t_2} & {n-2 \over t_1t_3} & \cdots & {n-2 \over t_1t_3}\\
{n-2 \over t_1t_2} & {n-1 \over t_2^2} & {n-2 \over t_2t_3} & \cdots & {n-2 \over t_2t_3}\\
{n-2 \over t_1t_3} & {n-2 \over t_2t_3} & {n-1 \over t_3^2}  & \cdots & {n-2 \over t_3^2} \\
\vdots & \vdots & \vdots & \ddots & \vdots \\
{n-2 \over t_1t_3} & {n-2 \over t_2t_3} & {n-2 \over t_3^2} & \cdots & {n-1 \over t_3^2}
\end{array}\), \quad
\mca_2 = b_2(n-4) \(
\begin{array}{ccccc}
0 & 0 & 0 & \cdots & 0\\
0 & {t_2^{n-6 \over 2} \over 2^{n/2}} & 0 & \cdots & 0\\
0 & 0 & {t_3^{n-6 \over 2} \over 2} & \cdots & 0\\
\vdots & \vdots & \vdots & \ddots & \vdots \\
0 & 0 & 0 & \cdots & {t_3^{n-6 \over 2} \over 2}
\end{array}\),\\
\mcb &= \(
\begin{array}{cccc}
b_1 & 0 & \cdots & 0\\
-{(n-2)b_2t_2^{n-4 \over 2} \over 2^{n/2}}& 0 & \cdots & 0\\
0 & 0 & \cdots & 0\\
\vdots & \vdots & \ddots & \vdots \\
0 & 0 & \cdots & 0
\end{array}\) \quad \text{and} \quad
\mcd = \(
\begin{array}{cccc}
{(n-2)b_2t_2^{n-2 \over 2} \over 2^{n/2}} & 0 & \cdots & 0\\
0 & -{b_2t_2^{n-2 \over 2} \over 2^{(n-2)/2}} & \cdots & 0\\
\vdots & \vdots & \ddots & \vdots \\
0 & 0 & \cdots & -{b_2t_2^{n-2 \over 2} \over 2^{(n-2)/2}}
\end{array}\).
\end{align*}
Note that here we used the notations $b_1$, $b_2$ and $f(\mt)$ to denote $b_1 := c_2 |\nabla a(\xi_0)|$, $b_2 := c_3 a(\xi_0) (n-2)$ and $f(\mt) := (t_1 \cdots t_k)^{2-n}$. Also, we applied the fact that the components of $\ms_0$ are $\sigma_1 = (1, 0, \cdots, 0)$ and $\sigma_i = 0$ for $i \ge 2$ so that
\[\mt(\ms_0) = \({b_2 \over b_1} \cdot \({t_2 \over 2}\)^{n-2 \over 2},\ t_2,\ {t_2 \over 2}, \cdots,\ {t_2 \over 2}\) \quad \text{and} \quad t_2 = 2^{n+2k-5 \over n+2k-3} \cdot \({b_1 \over b_2}\)^{2 \over n+2k-3}.\]
On the other hand, $\det\(\mca - \mcb \mcd^{-1} \mcb^t\) \ne 0$ guarantees the nondegeneracy of the matrix $D^2\wpe(\mt(\ms_0), \ms_0)$ and so it is sufficient to prove it. We see
\begin{multline*}
\mca - \mcb \mcd^{-1} \mcb^t \\
= \(
\begin{array}{ccccc}
[(n-1)-{2 \over n-2}]\lambda^{\beta_1}\wb_1^{\beta_2}\wb_2^{\beta_3} & (n-1)\lambda^{-\beta_4}\wb_1^{\beta_4}\wb_2^2 & (n-2)\lambda^2\wb_1^{\beta_4}\wb_2^2 & \cdots & (n-2)\lambda^2\wb_1^{\beta_4}\wb_2^2\\
(n-1)\lambda^{-\beta_4}\wb_1^{\beta_4}\wb_2^2 & (n-2)\lambda^{\beta_5}\wb_1^{\beta_6}\wb_2^{\beta_7} & 2(n-2)\lambda^{\beta_5}\wb_1^{\beta_6}\wb_2^{\beta_7} & \cdots & 2(n-2)\lambda^{\beta_5}\wb_1^{\beta_6}\wb_2^{\beta_7}\\
(n-2)\lambda^2\wb_1^{\beta_4}\wb_2^2 & 2(n-2)\lambda^{\beta_5}\wb_1^{\beta_6}\wb_2^{\beta_7} & 6(n-2)\lambda^{\beta_5}\wb_1^{\beta_6}\wb_2^{\beta_7} & \cdots & 4(n-2)\lambda^{\beta_5}\wb_1^{\beta_6}\wb_2^{\beta_7}\\
\vdots & \vdots & \vdots & \ddots & \vdots \\
(n-2)\lambda^2\wb_1^{\beta_4}\wb_2^2 & 2(n-2)\lambda^{\beta_5}\wb_1^{\beta_6}\wb_2^{\beta_7} & 4(n-2)\lambda^{\beta_5}\wb_1^{\beta_6}\wb_2^{\beta_7} & \cdots & 6(n-2)\lambda^{\beta_5}\wb_1^{\beta_6}\wb_2^{\beta_7}
\end{array}\)
\end{multline*}
where $\lambda := 2^{1 \over n+2k-3}$, $\wb_i := b_i^{1 \over n+2k-3}$ for $i = 1, 2$, $\beta_1 := n-2$, $\beta_2 := n+4k-4$, $\beta_3 := -2k+1$, $\beta_4 := n+2k-5$, $\beta_5 := -3n-4k+12$, $\beta_6 := n-6$ and $\beta_7 := 2k+3$.
Therefore
\[\det\(\mca - \mcb \mcd^{-1} \mcb^t\) = C \det\(
\begin{array}{cccccc}
(n-1)-{2 \over n-2} & n-1 & 1 & 1 & \cdots & 1\\
n-1 & n-2 & 1 & 1& \cdots & 1\\
2(n-2) & 2(n-2) & 3 & 2 & \cdots & 2\\
2(n-2) & 2(n-2) & 2 & 3 & \cdots & 2\\
\vdots & \vdots & \vdots & \vdots & \ddots & \vdots \\
2(n-2) & 2(n-2) & 2 & 2 & \cdots & 3
\end{array}\) = - C (n+2k-3) \ne 0\]
for some $C > 0$. Here the second equality can be derived from the induction on $k$.
 This concludes the proof.

\appendix
\section{Sketch of proofs of   Proposition \ref{prop_auxiliary} and Proposition \ref{prop_reduction}}\label{sec_sketch}
Here we sketch the proofs of   Proposition  \ref{prop_auxiliary} and Proposition \ref{prop_reduction}. We omit many details which can be found in the literature. We only highlight the steps where the effect of the anisotropic coefficient $a$ leads to  new estimates.
 Set $p = (n+2)/(n-2)$.

The first step is the estimate of the error term.
\begin{lemma}\label{lemma_error}
Let $R_\ep := \Pi_{\ep}\pe\(i_{\ep}^*\(|V_\ep|^{4 \over n-2}V_\ep\)-V_\ep\)$. Then
\[\|R_\ep\| = O\(\ep^{{n+6 \over 2(n+2)}\cdot{n-2 \over (n-1)+2(k-1)}}\) = o\(\ep^{{1 \over 2}\cdot{n-2 \over (n-1)+2(k-1)}}\)\]
where $o\(\ep^{n-2 \over (n-1)+2(k-1)}\) \cdot \ep^{-{n-2 \over (n-1)+2(k-1)}} \to 0$ as $\ep \to 0$ uniformly in $(\md, \ms) \in \Lambda_{d_0,s_0}$.
\end{lemma}
\begin{proof}
Write
\begin{align*}
R_\ep &= \Pi_{\ep}\pe\(i_{\ep}^*\(|V_\ep|^{p-1}V_\ep-\sik |P_\ep U_i|^{p-1}(-1)^{i+1}P_\ep U_i\)\)
+ \sik (-1)^{i+1} \Pi_{\ep}\pe\(i_{\ep}^*\(|P_\ep U_i|^{p-1}P_\ep U_i - U_i^p\)\)\\
&\ + \sik (-1)^{i+1} \Pi_{\ep}\pe\(i_{\ep}^*\(\nabla \log a \cdot \nabla PU_i\)\)
:= R_1 + R_2 +R_3.
\end{align*}
In a similar way to \cite[Section 3]{KP2}, we can check that $R_1,\ R_2 = O\(\delta_1^{n+6 \over 2(n+2)}\)$.
Also, \cite[Lemma A.10]{KP2} gives $R_3 = O(\delta_1^{1-\sigma})$ for any small $\sigma > 0$. 
Hence the lemma follows.
\end{proof}

The following lemma is crucial in the proof of Propositions \ref{prop_auxiliary} and
\ref{prop_reduction}.
\begin{lemma}\label{lemma_ortho}
Assume $i \le l$. Then it holds
\[(P\psi_i^j, P\psi_l^m)_{\ep} = \delta_{il}\delta_{jm} a(\xi_0) \tilde{c}_j \delta_i^{-2} + o\(\delta_i^{-2}\)\]
for some $\tilde{c}_0,\ \tilde{c}_1 = \cdots = \tilde{c}_n > 0$. Here $\delta_{jm} = 1$ if $j = m$ and $\delta_{jm} = 0$ otherwise.
\end{lemma}
\begin{proof}
By \eqref{eq_linearized}, we have
\[(P\psi_i^j, P\psi_l^m)_{\ep} = p \inte aU_i^{p-1}\psi_i^j\psi_l^m + p \inte aU_i^{p-1}\psi_i^j(P\psi_l^m - \psi_l^m) - \inte (\nabla a \cdot \nabla P\psi_i^j) P\psi_l^m.\]
The first term in the right-hand side in the above equality can be estimated as in \cite[Lemma A.5]{MP}, showing that
\[p \inte aU_i^{p-1}\psi_i^j\psi_l^m = \delta_{il}\delta_{jm} a(\xi_0) \tilde{c}_j \delta_i^{-2} + o(\delta_i^{-2}).\]
Also, the third term, which arises due to the anisotropy of $a$, can be handled by Green's representation formula of $\nabla P\psi_i^j$ and Young's inequality (see \cite[Lemma A.9]{KP2}). Indeed,
\begin{multline}\label{eq_ortho_1}
\inte |\nabla P\psi_i^j| |P\psi_l^m| \le \inte \inte {|P\psi_l^m(x)|U_i^p(y) \over |x-y|^{n-1}} dxdy\\
\le C \|P\psi_l^m\|_{L^q(\ome)} \big\|U_i^{p-1}\psi_i^j\big\|_{L^r(\ome)} = O\(\delta_l^{{n \over q} - {n \over 2}}\delta_i^{{n \over r} - {n+4 \over 2}}\)
\end{multline}
for some $C > 0$ and any $q, r > 1$ such that $q^{-1} + r^{-1} < (n+1)/n$.
Thus choosing $q^{-1} = r^{-1} = (n+1)/(2n) - \eta$ for some $\eta > 0$ sufficiently small, we conclude that the third term is $o\(\delta_i^{-2}\)$.
On the other hand, from \eqref{eq_expansion_1_1} in Lemma \ref{lemma_expansion_1}, we deduce that the second term satisfies
\begin{align*}
p \inte aU_i^{p-1}\psi_i^j(P\psi_l^m - \psi_l^m) &= \inte {\delta_i^{n \over 2} \over (\delta_i + |x - \xi_i|^2)^{n+2 \over 2}} \cdot O\(\delta_l^{n-4 \over 2} + {\ep^{n-2} \over \delta_l^{n \over 2} |x-\xi_0|^{n-2}}\)\\
&= O\(\delta_l^{n-4 \over 2}\delta_i^{n-4 \over 2}\) + O\(\ep^{n-2}\delta_l^{-{n \over 2}}\delta_i^{-{n \over 2}}\) = o\(\delta_i^{-2}\)
\end{align*}
if $j= 0$ and $m = 0$. The other cases (either $j \ge 1$ or $m \ge 1$) can be handled similarly. This completes the proof.
\end{proof}

\begin{proof}[Proof of Proposition \ref{prop_auxiliary}]
Let $L_{\ep}^{\md, \mt} \phi = \phi - p \cdot \Pi_{\ep}^{\pe}\(i^*\(|V|^{p-1}\phi\)\) $ for $\phi \in \(K_{\ep}^{\md, \mt}\)\pe$.
(Note that $K^{\pe}_{\ep}$ in \eqref{eq_space_perp} depends on the choice of $(\md, \mt)$. To emphasize it, we used the notation $\(K_{\ep}^{\md, \mt}\)\pe$.)
The main step for the proof is to show that there exist $\ep_0 > 0$ and $c > 0$ such that for each
$\ep \in (0, \ep_0)$ and $(\md, \mt) \in \Lambda_{d_0,t_0}$ the operator $L_{\ep}^{\md, \mt}$ satisfies $\big\|L_{\ep}^{\md, \mt} \phi\big\|_{\ep} \ge c\|\phi\|_{\ep}$ for all $\phi \in \(K_{\ep}^{\md, \mt}\)\pe$.

On the contrary, suppose that there are sequences of positive numbers $\{\ep_n\}_n$ and functions
$\{\phi_n\}_n$ such that $\phi_n \in \(K_{\ep_n}^{\md, \mt}\)\pe$, $\|\phi_n\|_{\ep_n} = 1$ for all $n \in \mathbb{N}$,
and $\big\|L_{\ep_n}^{\md, \mt}\phi_n\big\|_{\ep_n} \to 0$ and $\ep_n \to 0$ as $n \to \infty$.
If we denote $\psi_n = L_{\ep_n}^{\md, \mt}\phi_n$ and drop superscripts and subscripts, we obtain
\begin{equation}\label{eq_sketch_1}
\text{div}(a\nabla \phi) + p |V|^{p-1} \phi = \text{div}(a\nabla \psi) + \text{div}(a\nabla \tau)
\end{equation}
for some $\tau \in K$.
Then by writing $\tau = \sum_{l=1}^k\sum_{m = 0}^n c_{lm} P\psi_l^m$ with some $c_{lm} \in \mathbb{R}$ (see \eqref{eq_pdxn} and \eqref{eq_pdxj}) and multiplying \eqref{eq_sketch_1} by $a P\psi_i^j$ ($i = 1, \cdots, k$, $j = 0, \cdots, n$), we see that
\[p \inte a|V|^{p-1}\phi P\psi_i^j = - \sum_{l,m} c_{lm} (P\psi_i^j, P\psi_l^m)_{\ep}.\]
On the other hand, from \eqref{eq_space_perp}, \cite[(5.7)]{GMP} (see also \cite[(3.7)]{MP}) and the Young's inequality argument, we get
\[p \inte a|V|^{p-1}\phi P\psi_i^j = \inte (\nabla a \cdot \nabla P\psi_i^j) \phi + o\(\delta_i^{-1}\) = O\(\big\|\nabla P\psi_i^j\big\|_{L^{2n \over n+2}(\ome)}\) + o\(\delta_i^{-1}\) = o\(\delta_i^{-1}\).\]
Thus by considering Lemma \ref{lemma_ortho},
we conclude $c_{lm} = o(\delta_l)$ for each $l$ and $m$, or equivalently, $\|\tau\|_{\ep} = o(1)$.
Then by testing \eqref{eq_sketch_1} with $u := \phi - (\psi + \tau)$, we obtain $\inte a |V|^{p-1}u^2 \to 1/p$ as $\ep \to 0$.
However, the nondegeneracy of \eqref{eq_limit} implies $\inte a |V|^{p-1}u^2 \to 0$ as $\ep \to 0$ as shown in Step 3 and 4 of the proof of \cite[Lemma 5.1]{GMP} (again all the terms involved with $\nabla a$ are negligible).
Therefore a contradiction arises, which proves the main step.

Now, the Fredholm alternative implies that the inverse $L^{-1}$ of $L$ exists,
and by employing Lemma \ref{lemma_error} and the contraction mapping principle on the set $\left\{\phi \in K^{\pe} : \|\phi\|_{\ep} \le c\ep^{{n+6 \over 2(n+2)}\cdot{n-2 \over (n-1)+2(k-1)}}\right\}$ for some $c > 0$ small,
we can deduce that the operator
\[T(\phi) := L^{-1}(N(\phi) + R) \quad \text{where } N(\phi):= \Pi^{\pe} \(i^* \(|V+\phi|^{p-1}(V+\phi) - |V|^{4 \over n-2}V - p \cdot |V|^{p-1}\phi\)\) \]
has a fixed point, which satisfies \eqref{eq_auxiliary_1} and is a solution of \eqref{eq_main_auxiliary}.
Furthermore, the standard argument taking advantage of the implicit function theorem shows that the map $({\md, \ms}) \mapsto \phi_{\ep}^{\md, \ms}$ is $C^1$.
For a detailed treatment of these claims, we refer to \cite[Proposition 2.1]{GMP}.
\end{proof}

\begin{proof}[Proof of Proposition \ref{prop_reduction}]
Given $(\md, \ms) = ((d_1, \cdots, d_k), (\sigma_1, \cdots, \sigma_k)) \in \Lambda_{d_0, s_0}$, let $s$ be any of $d_1, \cdots, d_k$, $\sigma_{11}, \cdots, \sigma_{kn}$ where $\sigma_i = (\sigma_{i1}, \cdots, \sigma_{in}) \in \rn$ for each $i = 1, \cdots, k$.

Suppose $J'_{\ep}(\md, \ms) = 0$. Then, if we write $V = V_{\ep}^{\md, \ms}$ and $\phi = \phi_{\ep}^{\md, \ms}$, we get
\[0 = I'_{\ep}(V + \phi)(\ps V + \ps \phi) = \sum_{i,j} c_{ij}\(\big(P\psi_i^j, \ps V\big)_{\ep} - \big(\ps P\psi_i^j, \phi\big)_{\ep}\)\]
for some $c_{ij} \in \mathbb{R}$ ($i = 1, \cdots, k$ and $j = 0, \cdots, n$) where $\ps$ denotes $\dfrac{\partial}{\partial s}$. Hence the proof is done once we show that all $c_{ij}$'s are equal to 0.

By virtue of Lemma \ref{lemma_ortho}, we have
\[\big(P\psi_i^j, \ps V\big)_{\ep} = \begin{cases}
(-1)^{i+1} \left[\delta_{il} a(\xi_0) \tilde{c}_j d_i^{-1} \sigma_{ij}\delta_i^{-1} + o\(\delta_i^{-1}\)\right] &\text{if } s = d_l \text{ for } l = 1, \cdots, k,\\
(-1)^{i+1} \left[\delta_{il}\delta_{jm} a(\xi_0) \tilde{c}_j \delta_i^{-1} + o\(\delta_i^{-1}\)\right] &\text{if } s= \sigma_{lm} \text{ for } l = 1, \cdots, k,\ m = 0, \cdots, n
\end{cases}\]
where $\sigma_{i0} := 1$.
Besides one can deduce $\big\|\ps P\psi_i^j\big\|_{\ep} = O\(\delta_i^{-1}\)$ as in the proof of \cite[Lemma A.8]{KP2}, resulting $\big(\ps P\psi_i^j, \phi\big)_{\ep} = o\big(\big\|\ps P\psi_i^j\big\|_{\ep}\big) = o\(\delta_i^{-1}\)$.
These estimates are enough to draw that $c_{ij} = 0$, so our assertion is true.
\end{proof}

\end{document}